\documentclass[10pt]{amsart}
\usepackage{amssymb, amsmath}
\usepackage{amsthm}
\usepackage{amscd}
\usepackage{graphicx}
\newtheorem{theorem}{Theorem}
\newtheorem{lemma}[theorem]{Lemma}

\newtheorem{corollary}[theorem]{Corollary}

\theoremstyle{definition}
\newtheorem*{remark}{Remark}
\newtheorem*{acknowledgement}{Acknowledgement}

\title{Actions of higher rank, irreducible lattices on $\CAT(0)$ cubical complexes}
\author{T. T$\hat{\mathrm{a}}$m Nguy$\tilde{\hat{\mathrm{e}}}$n Phan}
\address{Department of Mathematics\\
5734 S. University Ave.\\
Chicago, IL 60637}
\email{ttamnp@math.uchicago.edu}

\DeclareMathOperator{\Lk}{Lk}

\DeclareMathOperator{\SO}{SO}

\DeclareMathOperator{\CAT}{CAT}

\DeclareMathOperator{\Cone}{Cone}

\DeclareMathOperator{\Fix}{Fix}

\DeclareMathOperator{\cube}{cube}

\def\Z{\mathbb{Z}}

\def\Q{\mathbb{Q}}

\input xy
\xyoption{all}
\begin{document}
\begin{abstract}
Let $\Gamma$ be an irreducible lattice of $\Q$-rank $\geq 2$ in a semisimple Lie group of noncompact type. We prove that any action of $\Gamma$ on a $\CAT(0)$ cubical complex has a global fixed point.   
\end{abstract}
\maketitle
\section{Introduction}
Let $\Gamma$ be an irreducible lattice of $\Q$-rank $r\geq 2$ in semisimple Lie groups of noncompact type. It is known that $\Gamma$ has property (FA), that is, any action of $\Gamma$ on a tree has a global fixed point. Property (FA) is generalized by Farb (\cite{Farb}), who proved that any action of $\Gamma$ on a $(r-1)$-dimensional $\CAT(0)$ complex has a global fixed point. 

The main theorem of this paper is a generalization of the fact that higher rank lattices $\Gamma$ have property (FA) in the sense that any action of $\Gamma$ on a $\CAT(0)$ cubical complex (of any dimension) has a global fixed point. This statement has been proved for groups with property (T) (\cite{Niblo}), which include higher rank, irreducible lattices $\Gamma$ in semisimple Lie groups $G$ (of noncompact type) that does not have a $\SO(n,1)$ or $SU(n,1)$ factor. 
\begin{theorem}[Main Theorem]
Let $\Gamma$ be an irreducible lattice of $\Q$-rank $\geq 2$ in semisimple Lie groups of noncompact type. Let $\Sigma$ be a $\CAT(0)$ cubical complex. Suppose that $\Gamma$ acts on $\Sigma$ by isometries preserving the cubulation of $\Sigma$. Then $\Gamma$ has a global fixed point in $\Sigma$.
\end{theorem}

\begin{remark}
We make the following remarks.
\begin{itemize}
\item[1)] The fixed point of $\Gamma$ does not have to be a vertex of $\Sigma$, but it is a vertex of the barycentric subdivision of $\Sigma$. 
\item[2)] The argument also applies to action without inversions of $\Gamma$ on $\CAT(0)$ polyhedral complexes whose cells are Coxeter polytopes.
\end{itemize}
\end{remark}

This note is a special case of part of a proof of a lemma in the author's paper on piecewise locally symmetric manifolds \cite{Tampwlocsym}. Grigori Avramidi has been insisting over more than a year that this part should be written up as a theorem on group actions by higher rank, irreducible lattices on $\CAT(0)$ cubical complexes. This paper is dedicated to him for his $\cube^{\cube}$ birthday.

\section{Proof of the main theorem}
By passing to the barycentric subdivision of $\Sigma$ we will assume that action of $\Gamma$ is without inversion (that is, if an element of $\Gamma$ preserves a cell of $\Sigma$, then it fixes that cell pointwise). Also, isometries of $\Sigma$ are semisimple since the translation distance of an isometry is discrete.
\subsection{Useful theorems on lattices and groups acting semisimply on $\CAT(0)$ spaces}

The following theorem on generation of higher rank lattices by nilpotent subgroups is due Farb (\cite[Proposition 4.1]{Farb}) and was proved in a more general setting with $\Q$ replaced by and algebraic number field $k$.
\begin{theorem}[\cite{Farb}]\label{gen by nil}
Let $\Gamma$ be an irreducible lattice of $\Q$ rank $r \geq 2$, and let $\Gamma$ act on a $\CAT(0)$ space $Y$. Then exists a collection of subgroups $\mathcal{C} = \{\Gamma_1, \Gamma_2, ..., \Gamma_{r+1}\}$ such that
\begin{enumerate}
\item The groups in $\mathcal{C}$ generate a finite index subgroup of $\Gamma$. 
\item Any proper subset of $\mathcal{C}$ generates a nilpotent subgroup $U$ of $\Gamma$.
\item There exists $m\in \Z^+$ so that for each $\Gamma_i \in \mathcal{C}$, there is a nilpotent group $N <C$ so that $r^m \in [N,N]$ for all $r \in \Gamma_i$. 
\end{enumerate}
\end{theorem}

Since we are dealing with group action on $\CAT(0)$ cubical complexes, the following theorem (\cite{Bridson}, \cite[Proposition 2.3]{Farb}) will prove to be useful.
\begin{theorem}\label{nilpotent action}
Let $N$ be a finitely generated, torsion-free, nilpotent group acting on a $\CAT{(0)}$ space $Y$ by semi-simple isometries. Then either $N$ has a fixed point or there is an $N$-invariant flat $L$ on which $N$ acts by translations and hence, factoring through an abelian group.
\end{theorem}

Given part (3) of Theorem \ref{gen by nil} and Theorem \ref{nilpotent action}, the following corollary (\cite{Bridson},\cite[Corollary 2.4]{Farb}) of Theorem \ref{nilpotent action} implies that each of group $\Gamma_i$'s in Theorem \ref{gen by nil} fixes a nonempty set $F_i$ in the $\CAT(0)$ space $Y$.

\begin{corollary}\label{fixsetofgenerators}
Let $N$ be a finitely generated, torsion-free, nilpotent group acting on a $\CAT{(0)}$ space $Y$ by semi-simple isometries. Then
\begin{itemize}
\item[1)] If $g^m\in [N,N]$ for some $m >0$, then $g$ has a fixed point.
\item[2)] If $N$ is generated by elements each of which has a common fixed point, then $N$ has a global fixed point. 
\end{itemize}
\end{corollary}

\begin{remark}
As pointed out in \cite{Farb}, it follows from Corollary \ref{fixsetofgenerators} above that for each such $U$ in Theorem \ref{gen by nil}, the set $\Fix(U)$ in $Y$ is nonempty. 
\end{remark}

\subsection{Proof of the main theorem}
By Theorem \ref{gen by nil}, the group $\Gamma$ is virtually generated by a collection $\mathcal{C}$ of nilpotent subgroups $N_1$, $N_2$, ..., $N_{r+1}$. In this proof we only need $\Gamma$ to be virtually generated by $3$ nilpotent groups that satisfies the conclusion of Theorem \ref{gen by nil}. So we write $\mathcal{C} = \{\Gamma_1, \Gamma_2, \Gamma_3\}$, for $\Gamma_1 = N_1$, $\Gamma_2 = N_2$, and $\Gamma_3$ is the (nilpotent) group generated by $N_3, N_4, ... N_{r+1}$. Observe that $\Gamma_i$'s satisfy the first two conclusions of Theorem \ref{gen by nil}, and the group generated by any two groups $\Gamma_i$ and $\Gamma_j$ fixes a nonempty set by the above remark.

For each $i = 1, 2, 3$, let $F_i$ be the set of points that is fixed by all elements of $\Gamma_i$, and we write $F_i = \Fix(\Gamma_i)$.  Let 
\[W_{ij} = F_i \cap F_j,\]
for $i, j = 1, 2, 3$. By the remark after Theorem \ref{gen by nil}, each $W_{ij}$ is nonempty. It is clear that $F_i$'s and thus $W_{ij}$'s are convex. We want to show that $\cap_{i=1,2,3}F_i \ne\emptyset$. Suppose the contrary, that \[\cap_{i=1,2,3}F_i = \emptyset.\]
\begin{figure}
\begin{center}
\includegraphics[height=90mm]{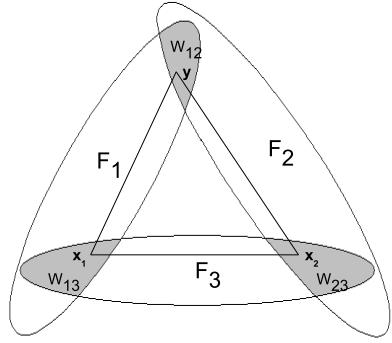}
\caption{The Fix sets $F_1$, $F_2$ and $F_3$ and their pairwise intersections.}
\end{center}
\end{figure}
Let $x_i \in W_{i3}$, for $i = 1,2$, be such that the distance $d(x_1,x_2) = d(W_{13},W_{23})$. Let $y \in W_{12}$. The geodesic $\gamma_{yx_1}$ (and $\gamma_{yx_2}$, respectively) lies in $F_1$ (and $F_2$, respectively) since $F_1$ and $F_2$ are convex. Without loss of generality, suppose that the angle (see \cite{Bridson} for the definition of \emph{angle})
\[\angle_{x_1} (y,x_2) < \pi/2,\] 
\textbf{Claim:} there is a simplex $S$ in $W_{13}$ containing $x_1$ and the angle 
\[\angle_{x_1}(S,x_1x_2) := \min_{s\in S, s\ne x_1} \angle_{x_1}(s,x_1x_2) < \pi/2.\]
Given the claim, it follows that $S$ contains some point $x_1'$ other than $x_1$ such that $d(x_1',x_2) < d(x_1,x_2)$, which is a contradiction to the choice of $x_1$ and $x_2$. Therefore,
\[\cap_{i=1,2,3}F_i \ne \emptyset.\] 
Therefore $\Gamma$ has a finite index group $\Gamma'$ that fixes a nonempty subset of $\Sigma$. Let $H$ be a finite index subgroup of $\Gamma'$ that is normal in $\Gamma$. Then the set $\Fix(H)$is nonempty. Since $\Gamma$ acts on $\Fix(H)$ with bounded orbit, it follows that $\Gamma$ has a global fixed point. We are left to prove the claim. 

Before proving the claim, we need the following definitions. A spherical simplex is \emph{all right} if each of its edge lengths is $\pi/2$. A piecewise spherical complex $Z$ is said to be \emph{all right} if all its simplices are all right.

\begin{proof}[Proof of Claim]
Let $Q$ be the cell of $\Sigma$ whose interior contains $x_1$. (If $x_1$ is a vertex of $\Sigma$, then $Q = \{x_1\}$.)
We consider the following cases.

\textbf{Case 1}: If  $Q$ is a point, then $Q = \{x_1\}$. Then the link of $\Lk(x_1, \Sigma)$ is a $\CAT(1)$ piecewise spherical, all right complex (by \cite[Lemma I.5.10]{Davis}). Apply Lemma \ref{fundamental lemma} below to $\Lk(x_1,\Sigma)$, we get a simplex $S \subset W_{13}$ containing $x_1$ such that the angle $\angle_{x_1}(x_2,S) < \pi/2$. The details of this argument is explain in the next paragraph.

Let $A$ (respectively, $B$) be the cell in $\Sigma$ that intersects nontrivially with $x_1y$ (respectively, $x_1x_2$). Let $A' = \Lk(x_1,A)$ and $B' = \Lk(x_1, B)$. Let $u = L \cap x_1y$ and $v = L \cap x_1x_2$. Then $d(u,v) < \pi/2$ in $L$. Let $U$ (respectively, $V$) be a face of $A'$ (respectively, $B'$) whose interior contains $u$ (respectively, $v$). By Lemma \ref{fundamental lemma}  there is a face $C \subset U\cap V$ such that $d_L(C, v) < \pi/2$.  Let $\alpha$ (respectively, $\beta$) be the convex hull of $x_1$ and $U$ (respectively, $V$). Then $\alpha \subset F_1$ and $\beta \subset F_3$. Let $S$ be the convex hull of $x_1$ and $C$. Therefore, the edge $S \subset W_{13}$ and has angle $< \pi/2$ with $x_1x_2$.  

\textbf{Case 2}: If $Q$ is not a point, then $L:= \Lk(x_1, \Sigma)$ is the spherical join $\Lk(x_1, Q)*\Lk(Q,\Sigma)$. \begin{itemize}
\item[2a)]If $\angle_{x_1}(x_1x_2, Q) < \pi/2$. Then pick $q\in Q$ such that $\angle_{x_1}(x_1x_2, x_1q) < \pi/2$. Since $Q \subset W_{13}$ (because $x_1$ is in the interior of $Q$ and $x_1 \in W_{13}$). We can let $S$ be the edge $x_1q \subset W_{13}$.

\item[2b)]Suppose that $\angle_{x_1}(x_1x_2, Q) \geq \pi/2$ (in which case we have equality). Then $x_1x_2$ intersects nontrivially with $\Lk(Q,\Sigma)$ at $v$. 

If $x_1y$ also intersects nontrivially with $\Lk(Q,\widetilde{T})$ at some point $u$, then argue as in the case $Q$ is a point using the fact that $\Lk(Q,\Sigma)$ is $\CAT(1)$ piecewise spherical, all right complex (\cite[Lemma I.5.10]{Davis}) and applying Lemma \ref{fundamental lemma}.

Suppose that $x_1y$ does not intersects nontrivially with $\Lk(Q,\Sigma)$. Then there is a point $q \in Q$ such that $x_1y$ intersects with $H := q* \Lk(Q,\Sigma)$ (the spherical joint of $q$ and $\Lk(Q, \Sigma)$) nontrivially. Let $u = x_1y \cap H$ and let $v = x_1x_2 \cap H$. Then $d_H(q,v) = \pi/2$ and $d_H(u,v) < \pi/2$. Therefore, the angle $\angle_q (u,v) < \pi/2$ by spherical geometry. By applying Lemma \ref{fundamental lemma} to the link $\Lk(q, H) = \Lk (Q, \Sigma)$, we deduce that there is  $C \subset \Lk(Q,\Sigma)$ such that the convex hull $S$ of $Q$ and $C$ is contained in $W_{13}$ and $d_L(C,v) < \pi/2$. Thus $\angle_{x_1} (x_2, S) < \pi/2$.   
\end{itemize} \end{proof}

\subsection{Proof of Lemma \ref{fundamental lemma}}
First we prove the following lemma.
\begin{lemma}\label{pre-fundamental lemma}
Let $L$ be a $\CAT(1)$, piecewise spherical, all right complex. Suppose that $d(x,y) < \pi/2$ for some $x, y \in L$. Let $X,Y$ be highest dimensional simplices of $L$ that contains $x$ and $y$ respectively. Then $X \cap Y \ne \emptyset$, and there is a point $p \in X\cap Y$ such that $\angle_p(x,y) < \pi/2$. 
\end{lemma}

\begin{proof}
Suppose that $X \cap Y  = \emptyset$. Let $\gamma$ be a  unit speed geodesic connecting $x$ and $y$ so that $\gamma (0) = x$ and $y \in \gamma ([0, \pi/2])$. Let $a < 0$ and $b>0$ so that $\gamma([a,b])$ is a connected component of the intersection of $\gamma$ with $X$. Let $v$ be a vertex of $X$, and let $A_v$ be the union of all cells containing $v$. Then $X \subset A_v$. We pick $v$ so that $d(v,\gamma(a)) > d(v, \gamma(b))$. 

Let $B_v$ be the closed $\pi/2$-ball centered at $v$. Then $B_v$ is isometric to the spherical cone $\Cone (\Lk (v, L))$ on the link of $v$ in $Z$. If $\gamma$ passes through $v$, then $d(v,y) < \pi/2$. Hence $ v \in Y$, and thus $v \in X \cap Y$, which contradicts the above assumption. So $\gamma$ does not pass through $v$. 

For each $t$ such that $\gamma(t) \in A_v$, let $s_t$ be the geodesic segment connecting $v$ passing through $\gamma(t)$ with length $\max (d(v,\gamma(t)), \pi/2)$.  Let $S$ be the surface defined as the union of all such $s_t$. In the same way as in \cite[Proof of Lemma I.6.4]{Davis}, the surface $S$ is  a union of triangles with common vertex $v$ glued together in succession along $\gamma$. Thus we can develop an $S$ along $\gamma$ locally isometric to $\mathbb{S}^2$, i.e. there is a map $f \colon S \longrightarrow \mathbb{S}^2$ that is a local isometry such that $f(v)$ is the North pole. Hence, $f(\gamma)$ is a geodesic in $\mathbb{S}^2$ that misses $f(v)$. Also, $f(S)$ contains the Northern hemisphere of $\mathbb{S}^2$.

The image $f(\gamma)$ cuts inside the region $f(S)$, which contains the Northern hemisphere $N$ of $\mathbb{S}^2$, so $f(\gamma)$ has length $\geq \pi/2$. Let $d$ be such that $f(\gamma (d))$ is where  $f(\gamma)$ exits $f(S)$. It is not hard to see that $\pi/2 > d > b > 0$. Since $d(v,\gamma(a)) > d(v, \gamma(b))$, it follows that $d(f(v),f(\gamma(a))) > d(f(v), f(\gamma(b)))$. Hence $d(f(\gamma(0)), f(\gamma(d))) > \pi/2$ be spherical geometry, which is a contradiction since $\gamma$ has unit speed. Therefore, $X \cap Y \ne \emptyset$.  
We can pick the point $p$ to be $v$. That $\angle_p(x,y) < \pi/2$ also follows from spherical geometry. 
\end{proof}

\begin{lemma}\label{fundamental lemma}
Let $L$ be a $\CAT(1)$,  piecewise spherical, all right complex. Let $X,Y$ be cells of $L$, and let $x \in X$, $y\in Y$. Let $X'$ (respectively, $Y'$) be a face of $X$ (respectively, $y$) whose interior contains $x$ (respectively, $y$). If $d(x,y) < \pi/2$, then $X'\cap Y'$ contains face $C$ such that $d(x, C) < \pi/2$.
\end{lemma}

\begin{proof}
We induct on the dimension of $L$. The base case is when $L$ has dimension $1$, in which case the lemma is obvious. Suppose that $L$ has dimension $>1$. By Lemma \ref{pre-fundamental lemma}, there is a point $p \in X\cap Y$ such that $\angle_p(x,y) < \pi/2$. Let $u$ (respectively, $v$) be the intersection of the geodesic ray $px$ (respectively, $py$) with the link $\Lk(p,L)$. Since $\angle_p(x,y) < \pi/2$, we have $d_{\Lk(p,L)}(u,v) < \pi/2$. Note that $\Lk(p,L)$ is also a $\CAT(1)$ piecewise spherical, all right complex (\cite[Lemma I.5.10]{Davis}).

Let $U'$ (respectively, $V'$) be a face in $\Lk(p,L)$ whose interior contains $u$ (respectively, $v$). Apply the induction hypothesis, there is a point $q \in \Lk(p,L)$ such that  $U'\cap V'$ contains face $D$ such that $d(u, D) < \pi/2$. 
Let $C$ be the face in $L$ that is spanned by $p$ and $D$. 
Then $C$ satisfies the condition in the lemma.
\end{proof}
 
\begin{acknowledgement}
I would like to thank Grigori Avramidi and Benson Farb for useful comments on earlier versions of this paper.
\end{acknowledgement}

\bibliographystyle{amsplain}
\bibliography{bibliography}

\end{document}